\documentclass[12pt, twoside, leqno]{article}

\usepackage{mystyle}

\usepackage{chngcntr}
\usepackage{apptools}

\AtAppendix{\counterwithin{theorem}{section}}

\newcommand\Tmbar{\clos{\set{T}}_{\bar{\idl{m}}}}
\newcommand\Tbar{\clos{\set{T}}}

\begin{document}


\baselineskip=17pt


\title{Filtrations of dc-weak eigenforms}

\author{Nadim Rustom\\
National Center for Theoretical Sciences\\ 
National Taiwan University\\
Room 203, Astronomy-Mathematics Building\\
No. 1, Sec. 4, Roosevelt Rd.\\Taipei City 106, Taiwan\\
E-mail: rustom.nadim@ncts.ntu.edu.tw}

\date{}

\maketitle


\renewcommand{\thefootnote}{}



\renewcommand{\thefootnote}{\arabic{footnote}}
\setcounter{footnote}{0}


\begin{abstract}
    The notions of strong, weak and dc-weak eigenforms mod $p^n$ was introduced and studied in \cite{CKW13}. In this work, we prove that there can be no uniform weight bound (that is, depending only on $p$, $n$) on dc-weak eigenforms mod $p^n$ of fixed level when $n \geq 2$. This is in contrast with the result of \cite{KRW16} which establishes a uniform weight bound on strong eigenforms mod $p^n$. As a step towards studying weights bounds for weak eigenforms mod $p^n$, we provide a criterion which allows us to detect whether a given dc-weak eigenform mod $p^n$ is weak.
\end{abstract}

\section{Introduction}
Throughout this paper, we denote by $p$ a prime number such that $p \geq 5$. For such a prime $p$, we fix an embedding $\injmor{\Qc}{\Qc[p]}$. We denote by $\sh{O}$ the ring of integers of a finite extension of $\Q[p]$, by $e$ the ramification index of $\sh{O}$, by $\pi$ a uniformiser of $\sh{O}$ and by $k$ the residue field of $\sh{O}$. We will consider modular forms and divided congruences of level $N =  p^{r}N_0$ (i.e., on $\Gamma_1(N)$) where $p \nmid N_0$ and $r \geq 1$. We will make precise the choice of these objects when needed.

Let $n\geq 1$ be an integer. For finite extensions $K$ and $L$ of $\Q[p]$ with $K \subset L$, with respective uniformisers $\pi_K$ and $\pi_L$ and respective ramification indices $e_K$ and $e_L$, we have inclusions
\[ \Zmod{p}{n} \subset \sh{O}_K/\pi^{e_K(n-1)+1} \subset \sh{O}_L/\pi^{e_L(n-1)+1}. \]
In fact, $t = e_K(n-1)+1$ is the smallest exponent $t$ such that $\Zmod{p}{n}$ embeds into $O_K/\pi^t$.
Thus we have a ring
\[ \Zmodbar{p}{n} = \varinjlim \sh{O}/\pi^{e(n-1)+1} \]
where the limit is over all rings of integers $\sh{O}$ of finite extensions of $\Q[p]$. Equivalently, we have
\[\Zmodbar{p}{n} = \Zc[p] / \{x \in \Zc[p] : v_p(x) > n - 1\} \] where $v_p$ is the normalised ($v_p(p) = 1$) valuation on $\Qc[p]$. Note that for $n = 1$, we have $\Zmodbar{p}{n} = \Fc[p]$. We think of $\Zmodbar{p}{n}$ as a topological ring with the discrete topology. For $\alpha, \beta \in \Zc[p]$, by $\alpha \pmod{p^n}$ and $\beta \pmod{p^n}$ we mean their images in $\Zmodbar{p}{n}$, and we will say that $\alpha \equiv \beta \pmod{p^n}$ if their images in $\Zmodbar{p}{n}$ agree. 

The work in the present paper falls within the theory of  mod $p^n$ modular Galois representations. The underlying motivation of that theory is the following. Let $G_{\Q} = \Gal(\Qc / \Q)$ be the absolute Galois group of $\Q$. Suppose we have a Hecke eigenform $f$ on $\Gamma_1(N)$, and let 
\[ \rho_{f,p}: \mor{G_{\Q}}{\GL(V)} \]be the $p$-adic Galois representation attached to it (i.e., $V$ is a $2$-dimensional $\Qc[p]$-vector space with a continuous action of $G_{\Q}$). The choice\footnote{Since these representations are continuous and the Galois group is compact, one can always find such a lattice.}  of a $G_{\Q}$-stable lattice $\Lambda$ in $V$ gives a representation
\[ \rho_{f, \Lambda, p } : \mor{G_{\Q}}{\GL_2(\Zc[p])}. \] For every integer $n \geq 1$, the reduction map $\surjmor{\Zc}{\Zmodbar{p}{n}}$ induces a mod $p^n$ Galois representation 
\[ \rho_{f,\Lambda,p,n}: \mor{G_{\Q}}{\GL_2(\Zmodbar{p}{n})} \]
attached to $f$. The Chenevier determinant (\cite{Che14}) arising from $\rho_{f, \Lambda, p, n}$ does not depend on the choice of $\Lambda$. If additionally $\Lambda\otimes \Fc[p]$ is irreducible then, by a result of Carayol (\cite{Car91}) together with the Chebotarev density theorem, $\rho_{f,\Lambda,p,n}$ is determined up to isomorphism by its traces at the Frobenius elements $\Frob_\ell$ for almost all $\ell$ and therefore does not depend on $\Lambda$. The representations $\rho_{f,\Lambda,p,n}$ are approximations of $\rho_{f,p}$, and one can hope to understand $\rho_{f,p}$ by first understanding its approximations $\rho_{f,p,n}$. 

The corresponding theory when $n=1$ has been of fundamental importance and has had numerous applications in arithmetic and geometry, e.g., in the proof of Fermat's Last Theorem and in the proof of Serre's modularity conjecture. This theory has a ``level" aspect and a ``weight" aspect. The most important part of the ``level" aspect is Ribet's level lowering theory, which translates statements about the structure of the restriction of $\rho_{f,p,1}$ to a decomposition subgroup $G_{\Q[\ell]} = \Gal(\Qc[\ell]/\Q[\ell])$ to a congruence mod $p$ between $f$ and some other Hecke eigenform of level $N / \ell$ whenever $\ell$ is a prime divisor of $N$ and $\ell \not = p$. Thus it is natural to try to understand $\rho_{f,p,n}$ by studying congruences mod $p^n$ between eigenforms. This approach has been developed to some extent, for example in \cite{Dum05} and \cite{Tsa09}. An application of level lowering mod $p^n$ to Diophantine problems is given in \cite{DY12}.

Related to this is the problem of stripping powers of $p$ from the level. Concretely, we have the following result.
\begin{theorem}[Buzzard, Hatada, Ribet]Let $p^r$ be the highest power of $r$ dividing $N$. Then there exists a Hecke eigenform $g$ of level $N/p^r$ such that $\rho_{f,p,1} \cong \rho_{g,p,1}$.
\end{theorem}
For $p \geq 3$, this was shown by Ribet (\cite{Rib94}, Theorem 2.1). For $p = 2$, this was shown independently by Buzzard (\cite{Buz00}) and by Hatada \cite{Hat01}.

In the course of their study of this problem, Chen, Kiming and Wiese (\cite{CKW13}) introduced the following three progressively weaker notions of eigenforms mod $p^n$: strong eigenforms, weak eigenforms, and dc-weak eigenforms, where ``dc" stands for ``divided congruence". Roughly, strong eigenforms are reductions mod $p^n$ of Hecke eigenforms in characteristic 0, while weak eigenforms are reductions of modular forms mod $p^n$ that behave like Hecke eigenforms after reduction. A dc-weak eigenform is defined in a similar manner by extending the action of Hecke operators to the space of divided congruences, first introduced by Katz (\cite{Kat75}), which consists of sums $\sum f_i$ of modular forms, not necessarily with integral coefficients, such that $\sum f_i$ does have integral $q$-expansion. These are also instances of Katz' generalised $p$-adic modular functions. The notions of strong, weak, and dc-weak coincide when $n = 1$, but they are distinct when $n > 1$. A discussion of how these notions compare is given at the end of Section 3. 

The authors of \cite{CKW13} show that one can attach to each eigenform in one of these classes a Galois representation with the natural desired properties (i.e., the trace of the image of Frobenius is related to the coefficients of the $q$-coefficients of the eigenform). 

The dc-weak eigenforms provide the context for a Ribet-type level lowering result mod $p^n$ for modular Galois representations modulo $p^n$. 

\begin{theorem}[\cite{CKW13}, Theorem 5] Let $f$ be a mod $p^n$ dc-weak eigenform of level $N p^r$. Assume that the residual representation $\rho_{f,p,1}$ is absolutely irreducible. Also assume $p \geq 5$. Then there exists a mod $p^n$ dc-weak eigenform $g$ of level $N$ such that $\rho_{f,p,n} \cong \rho_{g,p,n}$.
\end{theorem}

Furthermore, it is asked in \cite{CKW13} whether, in the case that $f$ is a strong eigenform, $g$ can also be taken to be a strong eigenform. To be able to investigate this question experimentally one needs some sort of bound on the weights in which strong eigenforms arise. This leads us to the ``weight" aspect, which we will discuss next.

When $n = 1$, we have the following classical ``weight bound" result of Jochnowitz\footnote{This is a generalisation of an unpublished argument of Tate and Serre in level $1$.}.

\begin{theorem}[\cite{Joc82}] Suppose $f$ is a (characteristic $0$) eigenform of level $N$. Then there exists a (characteristic $0$) eigenform $g$ of level $N$ and weight at most $p^2 + p$ such that $f \equiv g \pmod{p}$. 
\end{theorem}
In the above theorem, the congruence between the two eigenforms means that their $q$-expansions are coefficient-wise congruent mod $p$. This result shows that there are only finitely many congruence classes of mod $p$ eigenforms of level $N$, and hence only finitely many isomorphism classes of modular mod $p$ Galois representations of conductor $N$. This allows us to study mod $p$ Galois representations experimentally (see for example \cite{CG13}). Moreover, Jochnowitz' result was used by Edixhoven (\cite{Edi92}) to prove that the weaker form of Serre's modularity conjecture implies the stronger form of the same conjecture. 

In \cite{KRW16}, the author together with Kiming and Wiese showed, using an idea of Frank Calegari, the existence of a uniform weight bound on strong eigenforms mod $p^n$. More precisely, the result is the following.

\begin{theorem}[\cite{KRW16}]There exists a constant $\kappa(N, p, n)$ depending only on $N$, $p$, and $n$ such that any eigenform in characteristic $0$ of level $N$ is congruent mod $p^n$ to a modular form of level $N$ of weight at most $\kappa(N, p, n)$. 
\end{theorem}
In fact, it was conjectured in \cite{KRW16} that there are only finitely many congruence classes of strong eigenforms mod $p^n$ for any $n$. This is still an open problem and it is related to work of Buzzard (\cite{Buz05}) regarding slopes of modular forms.

It is therefore interesting to ask whether such a weight bound exists for weak and dc-weak eigenforms as well. The main purpose of this work is to show that no such bound exists for dc-weak eigenforms. In particular, if we denote by $D_w(R)$ the space of divided congruences of level $N$ with coefficients in $R$ and weight at most $w$ (see \cref{dcsection} for the definition), then we have the following result.

\newtheorem*{mainthm}{Theorem \ref{mainthm}}
\begin{mainthm} Suppose $p \geq 13$. Then for all integers $n\geq 2$, $d \geq 0$, there exists a mod $p^n$ dc-weak eigenform $f \in D_{w}(\Zmodbar{p}{n})$, such $w > d$ and $f \not \in D_{w'}(\Zmodbar{p}{n})$ for any $w' < w$. 
\end{mainthm}

The strategy is as follows. First, we introduce the space of divided congruences of level $N$ and coefficients in a $\Z[p]$-algebra $R$, and establish a base change property for this space. Then, we describe the action of the Hecke operators on divided congruences, and the Hecke algebra they generate. When $R$ is finitely generated as a $\Z[p]$-module, we establish a duality between the space divided congruences and the corresponding Hecke algebra. This gives us a base change property for the Hecke algebra $\set{T}$, and allows us to identify dc-weak eigenforms with coefficients in $R$ with certain $\Z[p]$-algebra homomorphisms $\mor{\set{T}}{R}$. 

Next, we introduce the notions of weight filtration and nilpotence filtration. The weight filtration of a divided congruence is the lowest weight in which it appears. The nilpotence filtration, first introduced by Khare (\cite{Kha98}), is defined in terms of the nilpotence of the action of the Hecke operators on the divided congruence under consideration. Nilpotence of Hecke operators is also the main tool used in \cite{Med15} to investigate Hecke algebras attached to mod $p$ modular forms. While the weight filtration is easier to handle on a computer, the nilpotence filtration is somehow easier and more natural to study from a theoretical point of view. For divided congruences mod $p^n$, we establish a comparison result between the weight filtration and the nilpotence filtration. In particular, we show that controlling one of these filtrations means controlling the other. Finally, we produce a situation in which the Hecke algebra $\set{T}$ is explicitly known (and is a finite integral extension of a power series ring). This allows us to write continuous maps $\mor{\set{T}}{R}$ with arbitrarily large nilpotence filtrations, concluding the proof. 

It would be interesting to answer the question of existence of weight bounds for weak eigenforms, but this seems to be a harder problem. The first difficulty lies in that the Hecke algebra corresponding to the space $S(R)$ (which consists of sums of classical modular forms with coefficients in $R$, see \cref{dcsection}) is more difficult to control. There is however another difficulty. Even if one knows explicitly the structure of the Hecke algebra of $S(R)$, the above construction, which works for dc-weak eigenforms, would a priori produce only eigenforms which are sums of modular forms of different weights. We should therefore find a criterion that ensures that a dc-weak eigenform in $S(R)$ is weak, i.e., that it lives in a single weight. We present a result in this direction, which is a strengthening of Proposition 25 in \cite{CKW13}.

\newtheorem*{dcisweak}{Theorem \ref{dcisweak}}
\begin{dcisweak} Let $R$ be a quotient of $\sh{O}$ that contains $\Zmod{p}{n}$. Let $f \in S(R)$ be a (normalised) dc-weak eigenform with coefficients in $R$. Then $f$ is weak if and only if the eigenvalue of $f$ under the action of $\gamma = 1+p \in \units{\Z[p]}$ lies in $\Zmod{p}{n}$. 
\end{dcisweak}

\cref{heckeops} closely follows \cite{BK15} (and its generalisation to higher levels in \cite{Deo17}), and contains slight generalisations of several basic lemmas used in \cite{BK15} from the case of coefficients in a finite field to the case of more general coefficient rings. We chose to make explicit all the details in order to provide a clear reference for these basic results for future use. In the Appendix, we provide a technical lemma for interchanging tensor product and projective limit which is needed to establish the base change property for the Hecke algebra.

\section{Divided congruences}\label{dcsection}
For a positive integer $i \geq 1$, we denote by 
\[ S_i(\sh{O}) \subset \sh{O}[\![q]\!] \] 
the $\sh{O}$-module of $q$-expansions of cusp forms of weight $i$ on $\Gamma_1(N)$ with coefficients in $\sh{O}$. Let
\[ S_{\leq w}(\sh{O}) = \sum_{i=1}^w S_i(\sh{O}) = \bigoplus_{i=1}^w S_i(\sh{O}), \]
\[ S(\sh{O}) = \sum_{i \geq 1} S_i(\sh{O}) = \bigoplus_{i\geq1} S_i(\sh{O}). \]
Note that these sums are direct by a classical argument (\cite{Miy89}, Lemma 2.1.1). If $R$ is an $\sh{O}$-algebra, let $S_i(R)$ (respectively $S(R)$) denote the $R$-module spanned by the image of $S_i(\sh{O})$ (respectively $S(\sh{O})$ via the canonical map $\mor{\sh{O}[\![q]\!]}{R[\![q]\!]}$. While it is true that 
$S_i(R) = S_i(\sh{O})\otimes R$, the sums
\[ S_{\leq w}(R) = \sum_{i=1}^w S_i(R), \indent S(R) = \sum_{i\geq 1} S_i(R) \]
are in general not direct since there are congruences between modular forms of different weights. For example, the image of $E_k - 1 \in S(\Z[p])$ in $S(\F[p])$ is 0 whenever $p \geq 5$ and $k \equiv 0 \pmod{p-1}$. Another way to see why we do not have a nice base change property for $S(\sh{O})$ is to note that the $\sh{O}$-module $\sh{O}[\![q]\!]/S(\sh{O})$ contains torsion elements. This motivates the following definition.
\begin{definition} The module of divided congruences with coefficients in $\sh{O}$ of weight at most $w$ and level $N$ is the module 
\[ D_w(\sh{O}) := \{f \in \sh{O}[\![q]\!] : \exists t \geq 0 \mbox{ such that } \pi^t f \in S_{\leq w}(\sh{O}) \}. \] 
The space of divided congruences is
\[ D(\sh{O}) :=  \bigcup_w D_w(\sh{O}). \]
If $R$ is an $\sh{O}$-algebra, let $D_w(R)$ (respectively $D(R)$) denote the $R$-module spanned by the image of $D_w(\sh{O})$ (respectively $D(\sh{O})$) in $R[\![q]\!]$. We call $D(R)$ the space of divided congruences of level $N$ with coefficients in $R$. 
\end{definition}
The space of divided congruences satisfies a nice base-change property.
\begin{proposition}\label{dcbasechange} Let $R$ be an $\sh{O}$-algebra. Then
\[ D_w(R) = D_w(\sh{O}) \otimes R \] and 
\[ D(R) = D(\sh{O}) \otimes R. \]
\end{proposition}
\begin{proof} It is enough to prove the first statement, as tensor products commute with direct limits. Consider the commutative diagram of $\sh{O}$-modules
\[ \xymatrix{ 0 \ar[r] & D_w(\sh{O}) \ar[r]\ar[d] & \sh{O}[\![q]\!] \ar[r]\ar[d] & \sh{O}[\![q]\!]/D_w(\sh{O}) \ar[r]\ar[d] & 0 \\
            T \ar[r] & D_w(\sh{O})\otimes R \ar[r] & R[\![q]\!] \ar[r] & \sh{O}[\![q]\!]/D_w(\sh{O})\otimes R \ar[r] &  0 } \]
obtained by tensoring the short exact sequence in the first row by $R$, where $T = \Tor[\sh{O}]{R}{\sh{O}[\![q]\!]/D_w(\sh{O})}$. Since $\sh{O}[\![q]\!]/D_w(\sh{O})$ is torsion-free over a PID, it is flat, and so $T = 0$. 
\end{proof}

\section{Hecke operators on divided congruences}\label{heckeops}
The main references for this section are \cite{Kat75}, \cite{Hid86}, and \cite{BK15}. 

For each $x \in \units{\Z[p]}$, there exists an operator $[x]$ acting on $D(\sh{O})$ as follows. If $f \in D(\sh{O})$ and $f = \pi^{-t} \sum_i f_i$ for some $t$ and some cusp forms $f_i \in S_{w_i}(\sh{O})$, then
\[ [x] f = \pi^{-t} \sum_i x^{w_i} f_i .\]
The fact that $[x]$ acts on $D(\sh{O})$ is not obvious and requires results of Katz. 

Also for each $a \in \units{\left(\Zmod{N}{}\right)}$, there exists an operator $\braket{a}$ on $D(\sh{O})$ defined as follows. On $S_w(\sh{O})$, $\braket{a}$ acts as the classical diamond operator on modular forms on $\Gamma_1(N)$. Since the diamond operator preserves $p$-integrality, its action extends to all of $D(\sh{O})$. 

The action of the Hecke operators $T_n$, $p \nmid n$, can also be extended from $S(\sh{O})$ to $D(\sh{O})$. In particular, for a prime $\ell \not = p$, the action of $T_\ell$ on the $q$-expansion of $f \in D(\sh{O})$ is given by
\[ a_n(T_\ell f) = \begin{cases} a_{\ell n}(f) \indent \mbox{if } \ell \nmid n\mbox{ or } \ell | N, \\ a_{\ell n}(f) + \inv{\ell}a_{n/\ell}([\ell]\braket{\ell}f) \indent \mbox{otherwise}.  \end{cases} \]
By introducing the operator $S_\ell$ for $\ell \not = p$ given by
\[ a_n(S_\ell f) = \begin{cases} a_{n}(\ell^{-2} [\ell]\braket{\ell}f) \indent \mbox{if } \ell \nmid N, \\ 0 \indent \mbox{otherwise}\end{cases} \]
and using the identities
\[ T_{mn} = T_m T_n \indent \mbox{when } (m,n) = 1 \]
and 
\[ T_{\ell^{n+1}} = T_\ell T_{\ell^n} - \ell S_\ell T_{\ell^{n-1}} \indent \mbox{for } n \geq 1 \]
we can define the Hecke operators $T_n$, $p \nmid n$, on $D(\sh{O})$. 

Moreover, since $p$ divides the level, the action of the $U$ operator also extends from $S(\sh{O})$ to $D(\sh{O})$. Thus $U$ acts on the $q$-expansion of $f \in D(\sh{O})$ by
\[ a_n(Uf) = \begin{cases} 0 \indent \mbox{if } p \nmid n, \\ a_{n/p}(f) \indent \mbox{if } p | n. \end{cases} \]
These identities imply that for any pair of integers $n, r \geq 0$, $p \nmid n$, and any $f \in D(\sh{O})$, we have $a_1(U^r T_n f) = a_{np^r}(f)$. 

Let $R$ be an $\sh{O}$-algebra. By  \cref{dcbasechange}, the Hecke operators $T_n$, $p \nmid n$, and $U$ induce operators on $D_w(R)$ which are compatible with the action on $q$-expansions. This allows us to define the Hecke operators $T_n$, $p \nmid n$, and $U$ on $D_w(R)$. We can now give a first definition of strong, weak and dc-weak eigenforms.
\begin{definition}\label{def1}\leavevmode\begin{enumerate}[1)]
\item A (normalised) dc-weak eigenform with coefficients in $R$, of level $N$ and weight at most $w$ is an element $f \in D_w(R)$ such that $T_n f = a_n(f)f$ for all $n \geq 1$, $p \nmid n$, and $Uf = a_p(f)f$. 
\item A dc-weak eigenform with coefficients in $R$, level $N$ and weight at most $w$ defined as in 1) is simply called weak of weight $w'$ if $f \in S_{w'}(R)$. That is, a weak eigenform is a dc-weak eigenform that can be found in a single weight.
\item A weak eigenform with coefficients in $R$, level $N$ and weight $w'$ defined as in 2) is called strong of weight $w''$ if $f$ is the image of a normalised eigenform $\tilde{f} \in S_{w''}(\tilde{\sh{O}})$ where $\tilde{\sh{O}}$ is a finite extension of $\sh{O}$.\end{enumerate}

When $R$ is a subring of $\Zmodbar{p}{n}$, we speak of a strong, weak or dc-weak eigenform mod $p^n$. 
\end{definition}
\begin{remark}The weights $w$, $w'$ and $w''$ appearing in \cref{def1} are not uniquely determined. Suppose for example that $p^n = 0$ in $R$. Due to the congruence relation $E_{p-1}^{p^{n-1}} \equiv 1 \pmod{p^n}$, a weak eigenform with coefficients in $R$ of weight $w'$ as in 2) is also of weight $w' + p^{n-1}(p-1)$ (as can be seen simply by multiplying it by $E_{p-1}^{p^{n-1}}$). \end{remark}
Let $\set{T}^{pf}_w(R)$ denote the $R$-subalgebra of endomorphisms of $D_w(R)$ generated by the operators $T_n$, $p \nmid n$ (the superscript ``pf"  is introduced in \cite{Deo17} and stands for ``partially full"). Let $\set{T}_w(R)$ denote the $R$-subalgebra of endomorphisms of $D_w(R)$ generated by $\set{T}^{pf}_w(R)$ and by $U$. We give $\set{T}^{pf}_w(R)$ and $\set{T}_w(R)$ the $\pi$-adic topology\footnote{Here, by abuse of notation, we use $\pi$ to denote the uniformiser of $\sh{O}$ as well as its image in $R$. If $\pi$ is nilpotent in $R$ then the corresponding topology is discrete.}. For $w' > w$ we have restriction maps
\[ \surjmor{\set{T}^{pf}_{w'}(R)}{\set{T}^{pf}_{w}(R)} \indent \mbox {and}\indent \surjmor{\set{T}_{w'}(R)}{\set{T}_{w}(R)}. \]
We let 
\[ \set{T}^{pf}(R) = \varprojlim_{w} \set{T}^{pf}_w(R)  \indent\mbox{and}\indent \set{T}(R) = \varprojlim_{w} \set{T}_w(R) \]
and we give these rings the corresponding projective limit topologies. 
\begin{remark}\label{noplevel} Recall that we assume $N = p^r N_0$ where $p \nmid N_0$ and $r \geq 1$. Just like above, we can define the space $D(N_0, \Z[p])$ of divided congruences of level $N_0$ with coefficients in $\Z[p]$, and the $\Z[p]$-algebra $\set{T}^{pf}(N_0, \Z[p])$ generated by the Hecke operators $T_n, p \nmid n$, acting on $D(N_0, \Z[p])$. By Proposition I.3.9 of \cite{Gou88}, $D(\Z[p])$ and $D(N_0, \Z[p])$ have the same closure in $\Z[p][\![q]\!]$ in the uniform convergence topology. It follows from this that there is a natural isomorphism $\set{T}^{pf}(\Z[p])\cong \set{T}^{pf}(N_0, \Z[p])$. See Corollary 13 of \cite{Deo17} and Corollary 13 of \cite{BK15} for further details. \end{remark}

We define an $R$-bilinear form 
\[ \mor{\set{T}_w(R) \times D_w(R)}{R} \]
\[ (T, f) \mapsto a_1(Tf). \]
This pairing satisfies the property that $a_1(T_n f) = a_n(f)$, and induces maps
\[ \tau_R: \mor{\set{T}_w(R)}{\Hom[R]{D_w(R)}{R}}, \] 
\[ \phi_R: \mor{D_w(R)}{\Hom[R]{\set{T}_w(R)}{R}}, \] where
\[ \tau_R(T)(f) = \phi_R(f)(T) = a_1(Tf). \]
\begin{proposition}\label{duality} The pairing defined above is perfect, i.e., the maps $\tau_R$ and $\phi_R$ are isomorphisms.
\end{proposition}
\begin{proof} The maps $\tau_R$ and $\phi_R$ are injective. To see that, suppose that $\phi_R(f)(T) = 0$ for all $T \in \set{T}_w(R)$. Then in particular, for any pair of integers $n, r \geq 0$, $p \nmid n$, we have $\phi_R(f)(U^r T_n) = a_1(U^r T_n f) = a_{np^r}(f) = 0$, hence \ $f = 0$. Similarly, if $\tau_R(T)(f) = 0$ for all $f \in D_w(R)$, then $\tau_R(T)(U^r T_n f) = a_{np^r}(Tf) = 0$ for all $n,r \geq 0$, $p \nmid n$, and all $f \in D_w(R)$. Thus $Tf = 0$ for all $f$ and hence $T = 0$. 

It remains to show that these maps are surjective. For $R = \sh{O}$, this follows from \cite{Hid86}, Proposition 2.1. In particular, $\set{T}_w(\sh{O})$ is the $\sh{O}$-dual of $D_w(\sh{O})$. Since $D_w(\sh{O})$ is finitely generated and torsion-free over a PID, it is free of finite rank. Thus $\set{T}_w(\sh{O})$ is also free of same rank as $D_w(\sh{O})$.

Suppose now that $R$ is an $\sh{O}$-algebra and let $\mor[\iota]{\set{T}_w(\sh{O})}{\set{T}_w(R)}$ be the natural map (sending Hecke operators to Hecke operators of the same name). This map gives rise to an $\sh{O}$-linear map 
\[\mor[\iota^{\ast}]{\Hom[R]{\set{T}_w(R)}{R}}{\Hom[\sh{O}]{\set{T}_w(\sh{O})}{ R}}. \]
Since $\set{T}_w(\sh{O})$ is free, we have isomorphisms of $\sh{O}$-modules
\[ \Hom[\sh{O}]{\set{T}_w(\sh{O})}{ R}\cong \Hom[\sh{O}]{\set{T}_w(\sh{O})}{\sh{ O}}\otimes_{\sh{O}} R   \]
\[  \cong D_w(\sh{O})\otimes_{\sh{O}} R \cong D_w(R)\]
(where we have used \cref{dcbasechange} for the last isomorphism). 

Let $\varphi\in \Hom[R]{\set{T}_w(R)}{R}$. Using the map $\iota^{\ast}$ and the isomorphisms just described, we can produce an element $f \in D_w(R)$. Now we only need to check that $\phi_R(f) = \varphi$. Since the image of $\iota$ generates $\set{T}_w(R)$ as an $R$-module, we only need to check that $\phi_R(f)(\iota(T)) = \varphi(\iota(T))$ for all $T \in \set{T}_w(\sh{O})$. This is done by unwrapping the definitions of all the maps involved.

The proof for $\tau_R$ is similar.
\end{proof}

We can now prove a base change property for the full Hecke algebra. 
\begin{corollary}\label{heckebasechange} Suppose $R$ is finitely generated\footnote{This would be case if $R$ is a quotient of a finite extension of $\sh{O}$.} as an $\sh{O}$-module. Then we have isomorphisms \[\set{T}_w(R) \cong \set{T}_w(\sh{O}) \otimes R\] and \[\set{T}(R) \cong \set{T}(\sh{O})\otimes R.\]  \end{corollary}
\begin{proof} The natural $\sh{O}$-linear map $\mor{\set{T}_w(\sh{O})}{\set{T}_w(R)}$ (obtained from  \cref{dcbasechange}) gives a surjective map $\surjmor{\set{T}_w(\sh{O})\otimes R}{\set{T}_w(R)}$. By  \cref{duality}, $\set{T}_w(\sh{O})\otimes R$ and $\set{T}_w(R)$ are both free $R$-modules of equal finite rank. This map is therefore an isomorphism. 

To prove the second statement, we need to pass to the limit. Even though tensor products do not in general commute with projective limits, we do have
\[ \varprojlim_w (\set{T}_w(\sh{O}) \otimes R) = \left(\varprojlim_w \set{T}_w(\sh{O}) \right)\otimes R \]
by \cref{tensorlimit} of the Appendix (which we can apply because each $\set{T}_w(\sh{O})$ is free, compact and Hausdorff).
\end{proof}

The pairing defined above induces upon taking limits a pairing
\[ \mor{\set{T}(R) \times D(R)}{R}. \]
We continue to denote the maps induced by $\tau_R$ and $\phi_R$ after taking limits by the same letters.
\begin{proposition}\label{pairinglimit}\leavevmode\begin{enumerate}[(a)]
\item The map $\tau_R$ induces an isomorphism
\[ \set{T}(R) \cong \Hom{D(R)}{R}. \]
\item The map $\phi_R$ is an injection
\[ \injmor{D(R)}{\Hom[cont]{\set{T}(R)}{R}}, \]
where $\Hom[cont]{\set{T}(R)}{R}$ is the set of $R$-linear maps $\mor{\set{T}(R)}{R}$ which are continuous in the projective limit topology on $\set{T}(R)$ and the $\pi$-adic topology on $R$. The image $\phi_R(D(R))$ consists of the $R$-linear maps $\mor{\set{T}(R)}{R}$ which factor through $\surjmor{\set{T}(R)}{\set{T}_w(R)}$ for some $w$. 
\item If $R$ is discrete, then $\phi_R$ is an isomorphism.
\end{enumerate}
\end{proposition}
\begin{proof} \begin{enumerate}[(a)]
\item The map $\tau_R$ induces an isomorphism since
\[ \set{T}(R) = \varprojlim_w \set{T}_w(R) \cong \varprojlim_w \Hom{D_w(R)}{R} \]
\[ \cong \Hom{\varinjlim_w D_w(R)}{R} = \Hom{D(R)}{R} \]
(where we have used \cref{duality} for the second isomorphism).
\item The projection map $\surjmor{\set{T}(R)}{\set{T}_w(R)}$ induces an injection
\[ \injmor{\Hom{\set{T}_w}{R}}{\Hom{\set{T}(R)}{R}}. \] By \cref{duality}, $\phi_R$ induces an injection
\[ \injmor{D(R) = \varinjlim_w D_w(R)}{\Hom{\set{T}(R)}{R}}. \]Thus $\phi_R$ identifies $D(R)$ with continuous $R$-linear maps from $\set{T}(R)$ to $R$ which factor through $\surjmor{\set{T}(R)}{\set{T}_w(R)}$ for some $w$. 


\item Suppose that $R$ is discrete and that $\varphi: \mor{\set{T}(R)}{R}$ is continuous. For each $w$, denote by $I_w$ the kernel of the projection $\surjmor{\set{T}(R)}{\set{T}_w(R)}$. Assume, for the sake of contradiction, that $I_w \not \subset \ker \varphi$ for all $w$. We can then find a sequence $\{x_w \in I_w\}$ such that $y_w = \varphi(x_w) \not = 0$ for all $w$. But $x_w \rightarrow 0$ and, since $\varphi$ is continuous, $y_w \rightarrow 0$. Since $R$ is discrete, the converging sequence $\{y_w\}$ is eventually constant. Thus $\varphi(x_w) = 0$ for some $w$, a contradiction. This finishes the proof.
\end{enumerate}
\end{proof}
\begin{proposition}\label{dcasmaps} Let $R$ be $\Z[p]$-algebra which is finitely generated as a $\Z[p]$-module. There is a one-to-one correspondence between divided congruences $f = \sum_n a_n q^n \in D(R)$ and $\Z[p]$-linear maps $\mor[\varphi]{\set{T}(\Z[p])}{R}$ such that  $\varphi$ factors through $\surjmor{\set{T}(\Z[p])}{\set{T}_w(\Z[p])}$ for some $w$ and $\varphi(T_n) = a_n$. This map $\varphi$ is a $\Z[p]$-algebra homomorphism if and only if $f$ is a normalised eigenform, i.e., $T_n f = a_n(f) f $ for all $n \geq 1$, $p \nmid n$, and $Uf = a_p(f)f$. 
\end{proposition}
\begin{proof} By tensor-hom adjunction and \cref{heckebasechange}, we have isomorphisms
\[ \Hom[R]{\set{T}_w(R)}{R} \cong \Hom[\mathbb{Z}_p]{\set{T}_w(\Z[p])}{R}, \]
\[ \Hom[R, cont]{\set{T}(R)}{R} \cong \Hom[\mathbb{Z}_p, cont]{\set{T}(\Z[p])}{R}. \]
Now using \cref{pairinglimit} establishes the first claim.

For the second statement, suppose first that $f$ is a normalised eigenform. Take any two operators $T, T' \in \set{T}(\Z[p])$ such that $Tf = \lambda f$ and $T'f = \lambda'f$. We have $\varphi(TT') = a_1(TT'f) = a_1(\lambda\lambda'f) = \lambda \lambda' = \varphi(T)\varphi(T')$.

Conversely, suppose that $\varphi$ is a homomorphism. Then for any $T\in\set{T}(\Z[p])$ and any $n\geq1$ such that $p \nmid n$, we have
\[a_n(Tf) = \varphi(T_n T) = \varphi(T_n)\phi(T) = \varphi(T)a_n(f) \]
and similarly
\[a_p(Tf) = \varphi(UT) = \varphi(U)\phi(T) = \varphi(T)a_p(f), \]
therefore $Tf = \varphi(T)f$. In particular, $f = \varphi(1)f$, and $\varphi$ is a ring homomorphism.
\end{proof}
Because of  \cref{dcasmaps}, we will simply write $\set{T}_w$ (respectively $\set{T}$, $\set{T}^{pf}_w$, $\set{T}^{pf}$) for $\set{T}_w(\Z[p])$ (respectively $\set{T}(\Z[p])$, $\set{T}^{pf}_w(\Z[p])$, $\set{T}^{pf}(\Z[p])$) and identify divided congruences with coefficients in $R$ with continuous maps $\mor{\set{T}}{R}$. 

We would like to compare the Hecke algebra of $D(\Z[p])$ congruences with the Hecke algebra of $S(\Z[p])$. Let $\sh{H}_w$ (respectively $\sh{H}^{pf}_w$) denote the $\Z[p]$-algebra generated by the restrictions of elements of $\set{T}_w$ (respectively $\set{T}^{pf}_w)$ to $S_{\leq w}(\Z[p])$. We let
\[ \sh{H}^{pf} = \varprojlim_{w} \sh{H}^{pf}_w  \indent\mbox{and}\indent \sh{H} = \varprojlim_{w} \sh{H}_w. \]
\begin{proposition}\label{teqh} We have (continuous) isomorphisms 
\[ \set{T} \cong \sh{H} \indent \mbox{and} \indent \set{T}^{pf} \cong \sh{H}^{pf}. \]
\end{proposition}
\begin{proof} For any weight $w$, we have surjective restriction maps
\[ \surjmor{\set{T}_w}{\sh{H}_w} \indent\mbox{and}\indent \surjmor{\set{T}^{pf}_w}{\sh{H}^{pf}_w}. \]
it is enough to show that these maps are isomorphisms. For this consider a $\Z[p]$-linear map $\mor[T]{D_w(\Z[p])}{D_w(\Z[p])}$ whose restriction to $S_{\leq w}(\Z[p])$ is $0$. Then $T$ factors through $D_w(\Z[p])/S_{\leq w}(\Z[p])$, which is a torsion module. However $D_w(\Z[p])$ is torsion-free, hence $T = 0$. 
\end{proof}

We now present a second definition of the notions of dc-weak, weak and strong eigenforms. 
\begin{definition}\label{def2}Suppose $R$ is finitely generated as a $\Z[p]$-module. \begin{enumerate}[1)]\item A dc-weak eigenform with coefficients  in $R$, of level $N$ and weight at most $w$ is a ring homomorphism $\mor[\varphi]{\set{T}}{R}$ which factors through $\set{T}_w$. 
\item A dc-weak eigenform with coefficients in $R$, of level $N$ and weight at most $w$ defined as in 1) is simply called weak of weight $w'$ if there exists a modular form $f = \sum a_n q^n \in S_{w'}(R)$ such that $\varphi(T_n) = a_n$ for all $n$ such that $p \nmid n$ and $\varphi(U) = a_p$. That is, a weak eigenform is a dc-weak eigenform that lives in a single weight.
\item A weak eigenform with coefficients in $R$, of level $N$ and weight $w$ defined as in 2) is called strong of weight $w''$ if $\varphi$ factors through $\set{T}_{w''}(\tilde{\sh{O}})$ for some finite extension $\tilde{\sh{O}}$ of $\sh{O}$.\end{enumerate}

When $R$ is a subring of $\Zmodbar{p}{n}$, we speak of a strong, weak or dc-weak eigenform mod $p^n$. 
\end{definition}

By \cref{dcasmaps}, \cref{def1} and \cref{def2} amount to the same when $R$ is finitely generated as a $\Z[p]$-module or a subring of $\Zmodbar{p}{n}$, and we can use them interchangeably. 

Clearly, we have inclusions
\[ \{ \mbox{strong eigenforms in $D_w(R)$}\} \subset \{ \mbox{weak eigenforms in $D_w(R)$} \}\]\[ \subset \{ \mbox{dc-weak eigenforms in $D_w(R)$} \}. \]
By composing a strong, weak, or dc-weak eigenform with the reduction map to the residue field $k$, we see that every strong, weak or dc-weak eigenform mod $p^n$ is a lift of some mod $p$ dc-weak eigenform. When $n = 1$, these sets coincide. More precisely, we have the following result. (see Lemme 6.11 of \cite{DS74}, and Lemma 16 of \cite{CKW13} for a proof in this setup).
\begin{lemma}[Deligne-Serre lifting lemma] A dc-weak eigenform mod $p$ of level $N$ and weight at most $w$ is also a strong eigenform mod $p$ of level $N$ and weight at most $w$. 
\end{lemma}
When $n \geq 2$, these inclusions might be proper. For example, at fixed weight and level, there are always only finitely many strong eigenforms mod $p^n$ at fixed weight and level, but there can be infinitely many weak eigenforms mod $p^n$ at the same weight and level. This was discovered by Calegari and Emerton in \cite{CE04}. See also Section 2.6 of \cite{KRW16} for a discussion of this example.

These inclusions might still be proper even if we remove the condition on the weight. There are examples of weak eigenforms of level $N$ which are not strong of level $N$ at any weight. For instance, as discussed in Section 2.5 of \cite{KRW16}, the mod $4$ reduction of the cuspform $E_4^3 \Delta + 2\Delta^3$ of weight $36$ and level $1$ is a weak eigenform mod $4$. But by a result of Hatada (\cite{Ha77}), every characteristic $0$ eigenform of level $1$ is congruent to $\Delta$, the unique cuspform of weight $12$ and level $1$. 

We still do not have any explicit example of a dc-weak eigenform which is not weak.

\section{Weight and nilpotence filtrations}
We keep the notation of the previous sections. Let $\set{T}^{sh}_w$ denote the $R$-subalgebra of endomorphisms of $D_w(R)$ generated by the operators $T_n$, $(n, pN_0) = 1$, and let \[\set{T}^{sh} = \varprojlim_w \set{T}^{sh}_w\] (the superscript ``sh" is introduced in \cite{Deo17} and stands for ``shallow"). The following theorem is well known.
\begin{theorem}\label{cnsl}The Hecke algebra $\set{T}$ is a $p$-adically complete Noetherian semilocal $\Z[p]$-algebra and factors into a product $\set{T} = \prod_{\idl{m}} \set{T}_{\idl{m}}$ where $\idl{m}$ runs over the finitely many maximal ideals of $\set{T}$ and $\set{T}_{\idl{m}}$ is the localisation of $\set{T}$ at $\idl{m}$. Each local component $\set{T}_{\idl{m}}$ is a complete Notherian local $\Z[p]$-algebra. 
\end{theorem}
\begin{proof}The maximal ideals of $\set{T}$ correspond to the $\Gal(\Fc[p]/\F[p])$-orbits of mod $p$ congruence classes of Hecke eigenforms of level $N$, of which there are only finitely many by a result of Jochnowitz (\cite{Joc82}). Let $\set{T}_{\idl{m}}$ be a local component of $\set{T}$ and denote by $\set{T}^{sh}_{\idl{m}}$ the corresponding local component of the shallow Hecke algebra. The theory of deformations of Galois (pseudo-) representations identifies $\set{T}^{sh}_{\idl{m}}$ with a quotient of the universal deformation ring of the corresponding mod $p$ Galois (pseudo-)representation. Thus $\set{T}^{sh}_{\idl{m}}$ is in particular topologically finitely generated as a $\Z[p]$-algebra and hence Noetherian (\cite{Kis09}). As $\set{T}_{\idl{m}}$ is topologically generated as a $\set{T}^{sh}_{\idl{m}}$-algebra by the operators $U$ and $T_{\ell}$ for $\ell | N_0$, it follows that $\set{T}$ is also Noetherian. 
\end{proof}

Let $\Tbar = \set{T}/p^n\set{T} \cong \set{T}(\Zmod{p}{n})$ (the last isomorphism follows by  \cref{heckebasechange}). We have a reduction map $\surjmor{\set{T}}{\Tbar}$. Thus by \cref{cnsl}, the mod $p^n$ Hecke algebra $\Tbar$ is also $p$-adically complete, Noetherian and semilocal with complete Noetherian local components. There is a one-to-one correspondence between the maximal ideals of $\Tbar$ and those of $\set{T}$ (obtained by reducing mod $p^n$).

For each $w \geq 0$, let $\Tbar_w = \set{T}_w(\Zmod{p}{n})$. Consider the projection map $\surjmor[pr_w]{\Tbar}{\Tbar_w}$. Since $\Tbar$ is semilocal, there exists an integer $c$ big enough so that, for any maximal ideal $\idl{m}$ of $\Tbar$, its image $pr_w(\idl{m})$ is a maximal ideal of $\Tbar_w$ whenever $w \geq c$. Each ring $\Tbar_w$ is $p$-adically complete, Noetherian and semilocal with complete Noetherian local components. Thus for each $w$ we can choose a labelling $\Tbar_w^{(i)}$, $1\leq i \leq r$, of the local components of $\Tbar_w$ such that the maximal ideal of $\Tbar_{w+1}^{(i)}$ maps to that of $\Tbar_w^{(i)}$. We get that
\[\Tbar \cong \varprojlim_{w\geq c} \Tbar_w = \varprojlim_{w\geq c}  \prod_{i=1}^r \Tbar_w^{(i)} = \prod_{i=1}^r \varprojlim_{w\geq c}  \Tbar_w^{(i)},\]
and so the local components of $\Tbar$ are given by
\[ \Tbar^{(i)} = \varprojlim_{w\geq c} \Tbar_w^{(i)} \]
for $1 \leq i \leq r$. 

There are two ways of viewing $\Tbar^{(i)}$ as a topological algebra. First, as $\Tbar^{(i)}_w$ is finite\footnote{Each $\Tbar_w$ is a finite ring (being a free $\Zmod{p}{n}$-module of finite rank by \cref{duality}).} for each $w$, we can give $\Tbar^{(i)}$ the profinite topology. Second, there is the natural topology generated by powers of the maximal ideal of $\Tbar^{(i)}_w$ which makes it a complete Noetherian semi-local ring. 

Let $R$ be a quotient of $\sh{O}$ in which $p^n = 0$. If we have a dc-weak eigenform $\mor[\varphi]{\set{T}}{R}$, then $\varphi$ must factor through $\Tbar_w^{(i)}$ for some $w$ and some $i$ because $R$ is a local ring with finite residue field. From now on, we let $\idl{m}$ denote the maximal ideal of $\set{T}$ corresponding to the local component $\Tbar^{(i)}$, and we write $\Tbar_{\idl{m}, w} = \Tbar_w^{(i)}$ and $\Tbar_{\idl{m}} = \Tbar^{(i)}$ and denote by $\clos{\idl{m}}$ the maximal ideal of $\Tbar_{\idl{m}}$. 

The ring $\Tbar_{\idl{m}, w}$ is Artinian, and it follows by Nakayama's lemma that there exists some integer $t \geq 0$ such that the image of $\clos{\idl{m}}^t$ in $\Tbar_{\idl{m}, w}$ is zero. In other words, the $\clos{\idl{m}}$-adic topology on $\Tbar_{\idl{m}}$ is finer than the profinite topology. But in our situation, these topologies are in fact the same.
\begin{proposition}\label{comptop} The profinite and $\idl{m}$-adic topologies on $\Tbar_{\idl{m}}$ coincide. 
\end{proposition}
\begin{proof} The profinite topology on $\Tbar_{\idl{m}}$ is generated by a basis of open neighbourhoods of $0$ consisting of the kernels of the projection maps $\surjmor{\Tbar_{\idl{m}}}{\Tbar_{w, \idl{m}}}$. The $\clos{\idl{m}}$-adic topology has as a basis of open neighbourhoods of $0$ the powers $\clos{\idl{m}}^t$ of the maximal ideal. Thus the statement of the proposition is equivalent to the statement that for any $w$ there exists some $t$ such that $\clos{\idl{m}}^t \subset \ker (\surjmor{\Tbar_{\idl{m}}}{\Tbar_{w, \idl{m}}})$ and for any $t'$ there exists a $w'$ such that $\ker (\surjmor{\Tbar_{\idl{m}}}{\Tbar_{w', \idl{m}}}) \subset \clos{\idl{m}}^{t'}$. A proof can be found in Medvedovsky's PhD thesis (\cite{Med15}, Proposition 2.9), and crucially uses the fact that $\Tbar_{\idl{m}}$ is Noetherian. \end{proof}
We will now define the weight and the nilpotence filtrations. 
\begin{definition} Let $\varphi$ be a dc-weak eigenform with coefficients in $R$. The weight filtration $\omega(\varphi)$ of $\varphi$ is the smallest weight $w$ such that $\varphi$ factors through ${\Tbar}_{\idl{m},w}$. The nilpotence filtration $\nu(\varphi)$ of $\varphi$ is the smallest $t$ such that $\varphi(\idl{m}^t) = 0$. \end{definition}
Because the profinite and the $\clos{\idl{m}}$-adic topologies on $\Tmbar$ agree, these two filtrations control each other, and we can relate them as follows.
\begin{corollary}\label{compfilt} There exists non-decreasing functions $\gamma, \delta: \mor{\set{N}}{\set{N}}$ such that, for any $\varphi$ as above, 
\[ \gamma(\nu(\varphi)) \geq \omega(\varphi) \indent\mbox{and}\indent \delta(\omega(\varphi)) \geq \nu(\varphi).\] 
\end{corollary}
\begin{proof} By \cref{comptop}, we can define $\gamma$ and $\delta$ by
\[ \delta(w) := \inf \{ t : \clos{\idl{m}}^t \subset \ker \left(\surjmor{\Tbar_{\idl{m}}}{\Tbar_{\idl{m},w}}\right) \}, \]
\[ \gamma(t) := \inf \{ w :  \ker \left(\surjmor{\Tbar_{\idl{m}}}{\Tbar_{\idl{m},w}}\right) \subset \clos{\idl{m}}^t \}. \]
\end{proof}
We conclude that in order to find a dc-weak eigenform $\varphi$ such that $\omega(\varphi)$ is very large, it is enough to make $\nu(\varphi)$ very large. 

\section{dc-weak eigenforms of large filtrations}
Let $\clos{f}$ be an eigenform of level $N = p^r N_0$ ($p \nmid N_0$ and $r \geq 1$) with coefficients in $\F[p]$ and $\clos{\rho}$ the Galois representation attached to it. We say that $\clos{\rho}$ is new away from $p$ if its away-from-$p$ conductor is $N_0$. Let $\mor[\varphi]{\set{T}}{\F[p]}$ be the associated algebra homomorphism, and $\idl{m} = \ker \varphi$ the associated maximal ideal. The natural map $\injmor{\set{T}^{pf}}{\set{T}}$ pulls back $\idl{m}$ to a maximal ideal $\idl{m}^\ast$ of $\set{T}^{pf}$, and we let $\set{T}^{pf}_{\idl{m}}$ denote the localisation of $\set{T}^{pf}$ at $\idl{m}^{\ast}$. The key ingredient that we will use is the following result.
\begin{theorem}\label{gmthm} Assume that $\clos{\rho}$ is irreducible, new away from $p$, and that its deformation problem is unobstructed. Then the Hecke algebra $\set{T}^{pf}_{\idl{m}}$ is isomorphic to a finite integral extension of $\Z[p][\![X,Y,Z]\!]$ (where $X$, $Y$, $Z$ are parameters). 
\end{theorem}
\begin{proof} Let $\set{T}_{\idl{m}}^{sh}$ be the shallow subalgebra of $\set{T}^{pf}_\idl{m}$. Combining a result of Yamagami (\cite{Yam03}), generalising\footnote{The Gouv\^ea-Mazur result assumes that $f$ has coefficients in $\Z[p]$ and is of non-critical slope. Yamamagi's result allows the coefficients of $f$ to lie in any finite extension of $\Z[p]$ and removes restrictions on slope.} results of Gouv\^ea-Mazur (\cite{GM98}), together with \cref{teqh}, we find that $\set{T}^{sh}_{\idl{m}}$ is a power series ring in three parameters. Since $\clos{\rho}$ is new away from $p$, Proposition 20 of \cite{Deo17} tells us that $\set{T}^{pf}_{\idl{m}}$ is a finite integral extension of $\set{T}_ {\idl{m}}^{sh}$. \end{proof}
\begin{corollary}\label{heckeps} Assume the conditions of \cref{gmthm}. Assume additionally that $U(\clos{f})=0$. Then the algebra $\set{T}_{\idl{m}}$ is isomorphic to a finite integral extension of $\Z[p][\![X,Y,Z,U]\!]$. 
\end{corollary}
\begin{proof} By imitating the proof of Proposition 17 of \cite{BK15} in the case of level $N$, we obtain a natural isomorphism of $\set{T}^{pf}_{\idl{m}}$-algebras $\set{T}^{pf}_{\idl{m}}[\![U]\!] \cong \set{T}_{\idl{m}}$. Combining this with \cref{gmthm} gives the required isomorphism.
\end{proof}

\begin{proposition}\label{constprop} Assume the conditions of \cref{heckeps}. Let $n \geq 2$. For any integer $d \geq 0$, there exists a dc-weak eigenform $\varphi_n$ mod $p^n$ of level $p$ lifting $\clos{f}$ and such that $\omega(\varphi_n) > d$. 
\end{proposition}
\begin{proof} By \cref{compfilt}, it suffices to exhibit a dc-weak eigenform mod $p^n$ whose nilpotence filtration is very large. Let $\sh{O}$ be the ring of integers of a finite extension of $\Q[p]$ with ramification index $e \gg 0$. By \cref{heckeps},  any continuous map
\[ \mor[\varphi_n]{\left(\Zmod{p}{n}\right)[\![X,Y,Z,U]\!]}{\sh{O}/\pi^{e(n-1)+1}\sh{O}} \]
can be extended to a map
\[ \mor[\tilde{\varphi}_n]{\set{T}_{\idl{m}}}{\Zmodbar{p}{n}}. \]
We can pick $\varphi_n$ such that $\pi \in \tilde{\varphi}_n(\idl{m})$ (for example, by sending one of the parameters to $\pi$). Then we would have that $\nu(\tilde{\varphi_n}) > e$, since $\tilde{\varphi}_n(\idl{m}^e)$ would contain $\pi^e = p$, and $p \not = 0$ because we assumed $n\geq 2$. Therefore we only need to pick a suitably large value for $e$, which we certainly can do.
\end{proof}
Thus we now have the main result which we will prove by exhibiting an explicit construction.
\begin{theorem}\label{mainthm} Suppose $p \geq 13$. Then for all integers $n\geq 2$, $d \geq 0$, there exists a mod $p^n$ dc-weak eigenform $f \in D_{w}(\Zmodbar{p}{n})$, such $w > d$ and $f \not \in D_{w'}(\Zmodbar{p}{n})$ for any $w' < w$. 
\end{theorem}
\begin{proof} Write $N = p^r N_0$ where $r \geq 1$ and $p \nmid N_0$. First suppose that $p \not = 691$, and let $\clos{\Delta}$ be the mod $p$ reduction of the unique normalised cuspform of level $1$ and weight $12$ and $\clos{\rho}$ the mod $p$ representation attached to it. Consider the mod $p$ eigenform $\theta^{p-1} \clos{\Delta}$. The mod $p$ representation attached to $\theta^{p-1} \clos{\Delta}$ is isomorphic to $\clos{\rho}$. Let $\set{T}_{\idl{m}}$ (respectively $\set{T}_{\idl{m}}'$) be the local component of the level $p^r$ (respectively level $N$) Hecke algebra corresponding to $\theta^{p-1}\clos{\Delta}$. 

The representation $\clos{\rho}$ is absolutely irreducible and its deformation problem is unobstructed (see Section 5.4 of \cite{Wes04} and Appendix F of \cite{Med15}). By \cref{noplevel}, $\clos{\rho}$ is also unobstructed in level $p^r$. Moreover, $U(\theta^{p-1}\clos{\Delta}) = 0$. Using the construction of \cref{constprop}, we can choose a finite extension of $\Q[p]$ with ring of integers $\sh{O}$ and arbitrarily large ramification index $e$ and produce a level $p^r$ dc-weak eigenform
\[ \mor[\varphi]{\set{T}_{\idl{m}}}{\sh{O}/\pi^{e(n-1)+1}\sh{O}} \]
such that $\pi \in \varphi(\idl{m})$. Composing $\varphi$ with the surjection 
$\surjmor{\set{T}_{\idl{m}}'}{\set{T}_{\idl{m}}} $ (induced by restriction) we obtain a level $N$ dc-weak eigenform \[\mor[\tilde{\varphi}]{\set{T}_{\idl{m}}'}{\sh{O}/\pi^{e(n-1)+1}\sh{O}}\] with arbitrarily large nilpotence filtration and hence arbitrarily large weight filtration.

If $p = 691$, we repeat the above argument replacing $\Delta$ by $\Delta_{16}$, the unique normalised cuspform of level $1$ and weight $16$.\end{proof}

\section{When is a dc-weak eigenform weak?}
In \cite{CKW13}, a sufficient condition is given for when a dc-weak eigenform is weak. More precisely, let $R$ be a quotient of $\sh{O}$ such that $\Zmod{p}{n} \subset R$. Suppose $f$ is a dc-weak eigenform with coefficients in $R$ which is in the image of $S(R)$, so that $f = \sum_i f_i$ where $f_i \in S_{i}(R)$. Then it is shown in \cite{CKW13}, Section 5.2, under the assumption that the $f_i$'s are linearly independent mod $p$, that $f$ is a weak eigenform. In this section, we will give another result in this direction which does not involve linear independence mod $p$.

\begin{lemma}\label{grouphecke} The homomorphism $\mor[\eta]{\units{\Z[p]}}{\End[R]{D(R)}}$, defined by $\eta(x)f = [x]f$, $f \in D(R)$, takes values in the subalgebra $\set{T}(R)$.
\end{lemma}
\begin{proof} See Lemma 5 of \cite{Deo17}. \end{proof}

\begin{theorem}\label{dcisweak}Let $R$ be a quotient of $\sh{O}$ that contains $\Zmod{p}{n}$. Let $f \in S(R)$ be a (normalised) dc-weak eigenform with coefficients in $R$. Then $f$ is weak if and only if the eigenvalue of $f$ under the action of $\gamma = 1+p \in \units{\Z[p]}$ lies in $\Zmod{p}{n}$. 
\end{theorem}
\begin{proof} If $f$ is weak of weight $w$, then $[\gamma]f = \gamma^w f$, so the eigenvalue of $f$ under the action of $\gamma$ lies in $\Zmod{p}{n}$.

We now show the converse. Write $f = \sum_i f_i$ where $f_i \in S_{i}(R)$. Since $f$ is a dc-weak eigenform mod $p^n$, $f$ is in particular an eigenvector for the action of a primitive $(p-1)$st root of unity $\zeta \in \units{\Z[p]}$. Since $a_1(f) = 1$ and $[\zeta^{p-1}]f = f$, we have
\[ [\zeta] f = \zeta^{\alpha} f \]
for some $\alpha \in \Zmod{(p-1)}{}$. 

Let $\gamma = 1+p \in \units{\Z[p]}$. Since $f$ is a dc-weak eigenform, $f$ is in particular an eigenvector for the action of $\gamma$, and we have assumed that
\[ [\gamma]f = \lambda f\]
where $\lambda \in \Zmod{p}{n}$. Since $f \in S(R)$ and $p^n R = 0$, we have
\[ [x]f = \sum_i x^i f_i = \sum_i f_i = f\] for all $x \in 1+p^n\Z[p]$, so in particular $[\gamma^{p^{n-1}}]f = f$. Since $a_1(f) = 1$, it follows that $\lambda^{p^{n-1}} = 1$. Thus $\lambda = (1+p)^\beta$ for some $\beta \in \Zmod{p}{n-1}$. 

The elements $\zeta$ and $\gamma$ topologically generate $\units{\Z[p]}$. That is, for an element $x \in \units{\Z[p]}$, we may write $x \equiv \zeta^t \gamma^s \pmod{p^n}$ for some $t \in \Zmod{(p-1)}{}$ and some $s \in \Zmod{p}{n-1}$. Then
\[ [x]f = [\zeta^t][\gamma^s]f = \zeta^{t\alpha}\gamma^{s\beta} f = x^w f \] where $w \in \Z$ is a non-negative integer chosen by the Chinese Remainder Theorem to satisfy
\begin{itemize}
\item $w \equiv \alpha \pmod{p-1}$, and
\item $w \equiv \beta \pmod{p^{n-1}}$.
\end{itemize}

We now argue exactly as in the second part of the proof of Proposition A1.6 in \cite{Kat75}. We present only a sketch, and we refer to \cite{Kat73} and \cite{Kat75} for details. Even though in these references Katz works over with the full level structure $\Gamma(N)$, the results carry over to $\Gamma_1(N)$-level structure. For details on that point, see Chapter 1 of \cite{Gou88}. 

Let $V_{n,0}$ denote the $R$-algebra of functions on the moduli space of pairs $(E/A, \iota_N)$ where $E/A$ is an elliptic curve over an $R$-algebra $A$ with ordinary reduction and $\iota_N$ is a level $N$ structure on $E$. Then for a specific $R$-algebra $A$, the $R$-algebra morphism $V_{n,0}\rightarrow A$ classifying pairs $(E/A, \iota_N)$ gives $A$ the structure of a $V_{n,0}$-algebra. Moreover, by Katz' moduli interpretation of divided congruences, $V_{n,0} \subset D(R)$. We let $A' = A\otimes_{V_{n,0}} D(R)$.

Since $p$ is nilpotent in $R$, we find by Katz' moduli interpretation that $f$ is a rule which to each situation
\[ \xymatrix{ (E/A, \alpha_N, \Phi) \ar[d] \\ \Spec(A)  } \]
where
\[ \begin{cases} A $ is an $R$-algebra,$ \\
(E, \alpha_N) $ is an elliptic curve with a level $N$ structure over $R, \\
\Phi $ is an isomorphism of formal groups $ \isomor{\widehat{E}}{\widehat{\set{G}}_m}, $ and $ \\
\end{cases} \]
assigns an element $f(E/A, \alpha_N, \Phi)$ of $A$ which depends only on the isomorphism class of $(E/A, \alpha_N, \Phi)$ and whose formation commutes with base change of $R$-algebras. 

Let $E/A$ be an elliptic curve with ordinary reduction, $\iota_N$ a level $N$ structure on $E$, and $\omega$ a nowhere-vanishing invariant differential. Since $E$ has ordinary reduction, there exists over $A'$ an isomorphism $\Phi: \isomor{\widehat{E}}{\widehat{\set{G}}_m}$, and thus 
\[ \omega = c\Phi^{\ast}\left( \frac{dT}{T+1} \right) \]
where $c \in A'$ and $dT/(T+1)$ is the canonical invariant differential on $\widehat{\set{G}}_m$. 

We define the element
\[ g(E/A, \iota_N, \omega) = c^{-w}f(E/A, \iota_N, \Phi) \in A'. \]
The definition of this element does not depend on the choice of the isomorphism $\Phi$, and this independence implies that
\[ g(E/A, \iota_N, \omega) \in A. \]
Thus the rule $g$ is a $p$-adic modular form over $R$, and its $q$-expansion is the same as the $q$-expansion of $f$. By Proposition 2.7.2 of \cite{Kat73}, there exists a true modular form $\tilde{g}$ of level $N$ and weight $w' \geq w$ whose $q$-expansion is the same as that of $g$, hence the same as that of $f$. This concludes the proof. 
\end{proof}

We immediately get the following result.
\begin{corollary} Suppose that $f \in S(\Zmod{p}{n})$. Then $f$ is a dc-weak eigenform if and only if it is weak.
\end{corollary}

With a little bit more work, we get the following result for dc-weak eigenforms mod $p^2$ with totally ramified coefficients.
\begin{corollary} Suppose that $\sh{O}$ is totally ramified of ramification index $e$. Let $R = \sh{O}/\pi^{e+1}$, and $f \in S(R)$. Then $f$ is a dc-weak eigenform if and only if $f$ is weak.
\end{corollary}

\begin{proof} Let $\gamma = 1+p$ and let $\lambda$ be the eigenvalue of $[\gamma]$ acting on $f$. Write $f = \sum_i f_i$ where $f_i \in S_i(R)$. Let $b_i = a_1(f_i)$. Since $f$ is a dc-weak eigenform, we have $\sum_i b_i = 1$. Moreover, we have $\sum_i \gamma^i b_i = \lambda$. Since $p^2R = 0$, we get
\[ \lambda = 1 + pc\]
for some $c \in R$. Since $\sh{O}$ is totally ramified, its residue field is $\F[p]$, hence there exists $c' \in \Z[p]$ such that $c \equiv c' \pmod{\pi}$. Multiplying by $p$, we get $pc \equiv pc' \pmod{\pi^{e+1}}$ hence $\lambda \in \Zmod{p}{2}$. We then apply \cref{dcisweak}.
\end{proof}

\appendix
\section{Tensor product and projective limit}

In the proof of  \cref{heckebasechange}, we needed to interchange tensor product and projective limit. For the sake of completeness, we provide a technical lemma that permits this interchange, since we could not find this result explicitly in the literature. 
\begin{lemma}\label{tensorlimit} Let $R$ be a topological ring, $B$ a finitely generated $R$-module, and 
\[\{\surjmor[\pi_{i,j}]{A_i}{A_j}\}_{i,j\geq 0}\] a projective system of compact Hausdorff $R$-modules such that the transition maps $\pi_{i,j}$ are continuous and surjective. Suppose further that $\Tor[R]{A_t}{B} = 0$ for all $t\geq 0$. Then
\[ \varprojlim_t \left(A_t \otimes_R B\right) \cong \left(\varprojlim_t A_t\right)\otimes_R B. \]
\end{lemma}
\begin{proof} Let $b_1, \ldots, b_n$ be generators of $B$. Let $A = \varprojlim_t A_t$ and put the projective limit topology on $A$. Then $A$ is a compact Hausdorff $R$-module. Let $I_t$ be the kernel of the projection map $\surjmor[\pi_t]{A}{A_t}$. Then for each $t$ we have inclusions $I_{t+1} \subset I_t$, and $\bigcap_t I_t = \{0\}$ since $A$ is Hausdorff. Moreover, we have for each $t$ a short exact sequence
\[\xymatrix{ 0\ar[r]&I_t\otimes_R B \ar[r]&A\otimes_R B \ar[r]& A_t\otimes_R B \ar[r]& 0 }\label{ses} \tag{$\ast$} \]
(since we are assuming that $\Tor[R]{A_t}{B} = 0$). Define a map of $R$-modules
\[ f:\mor{\prod_{i=1}^n A}{A\otimes_R B}, \]
\[ (a_1, \ldots, a_n) \mapsto \sum_{i=1}^n a_i \otimes b_i.\] We endow $A\otimes_R B$ with the topology coinduced by $f$, and $I_t\otimes_R B$ with the subspace topology. This makes $A\otimes_R B$ and $I_t\otimes_R B$ compact $R$-modules.

Applying the projective limit functor to \ref{ses}, we get an exact sequence
\[\xymatrix{0 \ar[r]& \varprojlim \left(I_t\otimes_R B\right)\ar[r] & A\otimes_R B \ar[r] & \varprojlim \left(A_t\otimes_R B\right) \ar[r] & \varprojlim^1 \left(I_t \otimes B\right). }\]
It is actually enough to show that $\varprojlim_t \left(I_t\otimes_R B\right) = 0$. For suppose for a moment that this is true. Then $\{I_t\otimes_R B\}_{t\geq 0}$ is a projective system of compact Hausdorff topological groups, and by Proposition 4.3 of \cite{McG95}, this implies that $\varprojlim^1 \left(I_t \otimes B \right)= 0$. 

Now we have \[ \varprojlim_t \left(I_t \otimes B\right) = \bigcap_t \left(I_t \otimes B\right). \]
Suppose $y \in \bigcap_t \left(I_t \otimes B\right)$. Then, for all $t$, there exists elements $x_1^{(t)}, \ldots, x_n^{(t)} \in I_t$ such that $y = \sum_{i=1}^n x_i^{(t)}\otimes b_i = f(x_1^{(t)}, \ldots, x_n^{(t)})$. Since $A$ is compact Hausdorff, for each $i$ we have that $x_i^{(t)}\rightarrow 0$ as $t \rightarrow \infty$. By continuity of $f$, this means that
\[ y = \lim_{t \rightarrow \infty} f(x_1^{(t)}, \ldots, x_n^{(t)}) = 0. \]\end{proof}

\subsection*{Acknowledgements}
This research was supported by VILLUM FONDEN through the network for Experimental Mathematics in Number Theory, Operator Algebras, and Topology while the author was a postdoc at the University of Copenhagen. The author would also like to thank the anonymous referee for the thorough reading of this paper and for the many valuable suggestions and comments that the referee provided.

\providecommand{\bysame}{\leavevmode\hbox to3em{\hrulefill}\thinspace}
\providecommand{\MR}{\relax\ifhmode\unskip\space\fi MR }
\providecommand{\MRhref}[2]{%
  \href{http://www.ams.org/mathscinet-getitem?mr=#1}{#2}
}
\providecommand{\href}[2]{#2}

\end{document}